\documentclass{amsart}
\usepackage{amsfonts}
\usepackage{amsmath}
\usepackage{amssymb}
\usepackage{graphicx}
\newtheorem{theorem}{Theorem}
\theoremstyle{plain}

\newtheorem{conjecture}{Conjecture}

\newtheorem{definition}{Definition}

\newtheorem{lemma}{Lemma}

\newtheorem{proposition}{Proposition}
\newtheorem{remark}{Remark}

\numberwithin{equation}{section}

\newcommand{\Le}[1]{\mathfrak{Le}_{#1}(x)}
\newcommand{\seq}[1]{\{#1\}_{k=0}^{\infty}}
\newcommand{\R}{\mathbb{R}}
\newcommand{\C}{\mathbb{C}}
\newcommand{\cms}{~ classical multiplier sequence~}

\newcommand{\g}[1]{g_{#1}^*}

\title{The non-existence of cubic Legendre multiplier sequences}

\author{Tam\'as Forg\'acs${}^{\dag}$, James Haley${}^{\ddag}$, Rebecca Menke${}^{\S}$, and Carlee Simon${}^{\star}$}
\thanks{Research partially supported by NSF grant DMS-1156273. Some of the work was completed while the first author was on sabbatical leave at the University of Hawai\textquoteleft i at Manoa, whose support he gratefully acknowledges.}

\begin{document}

\maketitle
\begin{abstract} The main result in this paper is the proof of the recently conjectured non-existence of cubic Legendre multiplier sequences. We also give an alternative proof of the non-existence of linear Legendre multiplier sequences using a method that will allow for a more methodical treatment of sequences interpolated by higher degree polynomials. \\
{\bf MSC 30C15, 26C10}\\
Keywords: Legendre multiplier sequences, symbol of a linear operator, coefficients of Legendre-diagonal differential operators 
\end{abstract}
\section{Introduction}
\smallbreak

Given a simple set of polynomials $Q=\seq{q_k(x)}$ and a sequence of numbers $\seq{\gamma_k}$, one can define the operator associated with $\seq{\gamma_k}$ as $T[q_k(x)]=\gamma_k q_k(x)$ for $k=0,1,2,\ldots$ and extend its action to $\R[x]$ linearly. Our work in this paper concerns such operators when $Q$ consists of the Legendre polynomials.
\bigbreak

\begin{definition}
The Legendre polynomials $\Le{k}$ are defined by the following generating relation
\begin{equation*}
\displaystyle{\frac{1}{\sqrt{1-2xt+t^2}}} = \displaystyle \sum \limits _{k=0}^\infty \mathfrak{Le}_k(x)t^k,
\end{equation*}
where the square root denotes the branch which goes to 1 as $t \to 0$. 
\end{definition}

\begin{definition}
A sequence of real numbers, $\seq{\gamma_k}$, is a Legendre multiplier sequence if $\displaystyle{\sum \limits _{k=0}^n a_k\gamma_k\mathfrak{Le}_k(x)}$ has only real zeros whenever $\displaystyle{\sum \limits _{k=0}^n a_k\mathfrak{Le}_k(x)}$ has only real zeros. We define $Q$-multiplier sequences for any basis $Q$ of $\R[x]$ analogously. If $Q$ is the standard basis, the associated multiplier sequences are called classical multiplier sequences (of the first kind).
\end{definition}
 We point out that every sequence of the form $\{0,0,0,\ldots, a,b,\ldots, 0,0,0,\ldots\}$, where $a,b \in \R$, is a Legendre multiplier sequence. The literature calls such sequences \emph{trivial}. In addition to these, there is an abundance of \emph{non-trivial} Legendre multiplier sequences (see \cite{Blakeman} for examples). Thus, the problem of characterizing polynomials which interpolate Legendre multiplier sequences is a meaningful one, and it fits well into the landscape of current research in the theory of multiplier sequences (see for example \cite{Blakeman}, \cite{bo}, \cite{tomandrzej2}, \cite{yoshi}). The present paper contributes to this line of inquiry by settling a conjecture on the non-existence of cubic Legendre multiplier sequences (Open problem (1) in \cite{Blakeman}). In addition, we give a new proof of the non-existence of linear Legendre multiplier sequences, which is more methodical than the educated hunt for test polynomials whose zeros fail to remain real after having been acted on by a linear sequence. 
\\  
\indent The rest of the paper is organized as follows. In Section \ref{background} we present a number of known results which are relevant to our investigations. Section \ref{linears} exhibits a new proof of the non-existence of linear Legendre multiplier sequences (Proposition 2 in \cite{Blakeman}) using a theorem of Borcea and Br\"and\'en. Our method exploits the fact that one does not need to have full knowledge of all coefficient polynomials $T_k(x)$ of a linear operator $\displaystyle{T=\sum_{k=0}^{\infty}T_k(x)D^k}$ in order to decide whether or not $T$ is reality preserving. Section \ref{cubics} contains the main result, Theorem \ref{nocubics}, which establishes the non-existence of cubic Legendre multiplier sequences. We conclude with a section on open problems.

\vspace{.25in}
\section{Background}\label{background}
\smallbreak
 
Central to the theory of (classical) multiplier sequences is the Laguerre-P\'olya class of real entire functions, which we denote by $\mathcal{L-P}$. We recall the definition here, along with a recent theorem characterizing this class as precisely those real entire functions which satisfy the generalized Laguerre inequalities.
\begin{definition}
A real entire function $\displaystyle{\varphi(x)=\sum_{k=0}^{\infty} \frac{\gamma_k}{k!}x^k}$ is said to belong to the Laguerre-P\'olya class, written $\varphi \in \mathcal{L-P}$, if it can be written in the form
\[
\varphi(x)=c x^m e^{-ax^2+bx} \prod_{k=1}^{\omega} \left(1+\frac{x}{x_k} \right) e^{-x/x_k},
\]
where $b,c \in \R$, $x_k \in \R \setminus \{ 0\}$, $m$ is a non-negative integer, $a\geq 0$, $0 \leq \omega \leq \infty$ and $\displaystyle{\sum_{k=1}^{\omega} \frac{1}{x_k^2} < +\infty}$. If $\gamma_k \geq 0$ for all $k=0,1,2,\ldots$ we say that $\varphi \in \mathcal{L-P}^+$. 
\end{definition}
Csordas and Vishnyakova recently completed the following characterization of the class $\mathcal{L-P}$.
\begin{theorem}\label{cvish}(\cite[Theorem 2.9]{cvarga} and \cite[Theorem 2.3]{cv}) Let $\varphi(x)$ denote a real entire function, $\varphi(x) \not\equiv 0$. Then $\varphi \in \mathcal{L-P}$ if and only if for all $n \in \mathbb{N}_0$ and for all $x \in \R$
\[
L_n(x,\varphi):=\sum_{j=0}^{2n} \frac{(-1)^{j+n}}{(2n)!}\binom{2n}{j} \varphi^{(j)}(x)\varphi^{(2n-j)}(x) \geq 0.
\]
\end{theorem}
We shall make use of this theorem in Section \ref{linears} when we reprove the non-existence of linear Legendre multiplier sequences. Since $\mathcal{L-P}$ is exactly the class of real entire functions which are locally uniform limits on $\C$ of real polynomials with only real zeros (see \cite[Ch.VIII]{levin}, or \cite[Satz 3.2]{obre}), it is closed under differentiation. Thus if $\varphi \in \mathcal{L-P}$, then 
\[
L_1(x,\varphi^{(k)}(x)) \geq 0 \qquad \forall k \in \mathbb{N}_0.
\]

In 1914, P\'olya and Schur completely characterized classical multiplier sequences. Their seminal theorem  maintains relevance in the setting of Legendre multiplier sequences, since every Legendre multiplier sequence must also be a classical multiplier sequence (see \cite[Theorem 8]{Blakeman} together with \cite[Proposition 118]{andrzej}). We note that if $\seq{\gamma_k}$ is a classical multiplier sequence, then either $\seq{\gamma_k}$, $\seq{-\gamma_k}$, $\seq{(-1)^k\gamma_k}$ or $\seq{(-1)^{k+1}\gamma_k}$ is a sequence of non-negative terms \cite[p.90]{Polya}.  Since $\seq{-1}$ and $\seq{(-1)^k}$ are both classical multiplier sequences, it suffices to consider only sequences of non-negative terms when characterizing classical multiplier sequences.

\begin{theorem}\label{Polya}
(P\'olya - Schur, \cite{Polya})
Let $\seq{\gamma _k}$ be a sequence of non-negative real numbers. The following are equivalent:
\begin{enumerate}
\item $\seq{\gamma _k}$ is a classical multiplier sequence.
\item For each $n$, the polynomial $T[(1+x)^n]:= \displaystyle \sum_{k=0}^{n} {n\choose k} \gamma _k x^k\in \mathcal{L-P^+}$.
\item $T[e^x]:= \displaystyle \sum_{k=0}^{n} \frac{\gamma _k}{k!} x^k\in \mathcal{L-P^+}$.
\end{enumerate}
\end{theorem}
Similarly to the classical setting, we may consider only sequences of non-negative terms when investigating (linear and cubic) Legendre multiplier sequences, by virtue of $\seq{(-1)^k}$ also being a Legendre multiplier sequence (\cite[Theorem 12]{Blakeman}).

\indent We conclude this section by a theorem of Borcea and Br\"and\'en, which characterizes reality preserving linear operators $T: \R[x] \to \R[x]$ in terms of their symbol $G_T(x,y)$. In order to be able to state their result (cf. Theorem \ref{symbol}), we need to make the following definitions.
\begin{definition}
The symbol of a linear operator $T:\mathbb{R}[x]\rightarrow\R[x]$ is the formal power series defined by
\begin{equation*}
G_T(x,y):= \sum_{n=0}^{\infty}\frac{(-1)^nT(x^n)}{n!}y^n.
\end{equation*}
\end{definition}
\begin{definition} A real polynomial $p \in \R[x,y]$ is called stable, if $p(x,y) \neq 0$ whenever Im$(x)>0$ and Im$(y)>0$. The Laguerre-P\'olya class of real entire functions in two variables, denoted by $\mathcal{L-P}_2(\R)$, is the set of real entire functions in two variables, which are locally uniform limits in $\mathbb{C}^2$ of real stable polynomials. 
\end{definition}

\bigbreak

\begin{theorem}\label{symbol}
(Borcea and Br\"{a}nd\'{e}n, 2009)
A linear operator $T:\mathbb{R}[x]\rightarrow\mathbb{R}[x]$ preserves the reality of zeros if and only if
\begin{enumerate}
\item The rank of $T$ is at most two and $T$ is of the form
\begin{equation*}
T(P) = \alpha(P)Q+\beta(P)R,
\end{equation*}
where $\alpha,\beta:\mathbb{R}[x]\rightarrow\mathbb{R}$ are linear functionals and $Q+iR$ is a stable polynomial, or;
\item $G_T(x,y) \in \mathcal{L-P}_2(\mathbb{R})$, or;
\item $G_T(-x,y) \in \mathcal{L-P}_2(\mathbb{R})$.
\end{enumerate}
\end{theorem}
\indent  In the remainder of this paper we follow the literature by using the notation $T=\seq{\gamma_k}$ to indicate the dual interpretation of a sequence as a linear operator and vica versa.
\section{Linear Legendre sequences}\label{linears}
We now reprove the non-existence of linear Legendre multiplier sequences (see \cite[Proposition 2]{Blakeman}). Although the result is known, our proof is novel, and has the promise of being suitable for use when investigating $Q$-multiplier sequences in larger generality. The following definition and three lemmas serve as setup for Theorem \ref{linear}. 
\begin{definition}\label{Defhyper} We define a generalized hypergeometric function by
\begin{equation}\label{hyperdef}
{}_pF_q\left[\begin{array}{rrrrr} a_1, &a_2,&\ldots&a_p;& \\  & & & &x\\ b_1,&b_2,&\ldots,&b_q; \end{array} \right]:=1+\sum_{n=1}^{\infty} \frac{\prod_{i=1}^p (a_i)_n}{\prod_{j=1}^q (b_j)_n}\frac{x^n}{n!},
\end{equation}
where $(\alpha)_n=\alpha(\alpha+1) \cdots (\alpha+n-1)$ denotes the rising factorial. 
\end{definition}
The convergence properties of the series on the right hand side of equation (\ref{hyperdef}) are discussed in detail in \cite[Ch.5]{rainville}. Here we mention that if $p=3$ and $q=2$, then the series is absolutely convergent on $|x|=1$ if
\[
\Re\left(\sum_{j=1}^q b_j-\sum_{i=1}^p a_i\right)>0.
\]  
\begin{lemma}\label{hypergeometric} The generalized hypergeometric function
\[
{}_3F_2\left[\begin{array}{rrrr} -\frac{1}{2}, &-n, &\frac{1}{2}+n;& \\ & & &-x\\ & \frac{1}{4},& \frac{3}{4}; \end{array} \right]
\]
converges at $x=-1$ and satisfies the equation
\[
{}_3F_2\left[\begin{array}{rrrr} -\frac{1}{2}, &-n, &\frac{1}{2}+n;& \\ & & &1\\ & \frac{1}{4},& \frac{3}{4}; \end{array} \right]=4n+1 \qquad \forall n \geq 1.
\]
\end{lemma}
\begin{proof} Convergence at $x=-1$ follows from the fact that 
\[
\Re\left(\frac{1}{4}+\frac{3}{4}-\left(-\frac{1}{2}-n+\frac{1}{2}+n \right) \right)=1>0,
\] 
together with the remark after Definition \ref{Defhyper}.
The rest of the claim follows directly from an application of Theorem 30 in \cite{rainville}, which states that for non-negative integers $n$, and $a,b$ independent of $n$ we have
\[
{}_3F_2\left[\begin{array}{rrrr} \frac{1}{2}+\frac{1}{2}a-b, &-n, &a+n;& \\ & & &1\\ & 1+a-b,& \frac{1}{2}a+\frac{1}{2}; \end{array} \right]=\frac{(b)_n}{(1+a-b)_n}.
\]
Setting $\displaystyle{a=\frac{1}{2}}$ and $\displaystyle{b=\frac{5}{4}}$ gives the required result.
\end{proof}
\begin{lemma}\label{appendix} Let $n \in \mathbb{N}^{\geq 1}$, and define
\[
\Psi_n(x):=\sum_{j=1}^n \binom{n}{j} \frac{(2j-2)!}{(j-1)!} \frac{\left(\frac{1}{2}+2j \right)_{n-j}}{\left( \frac{1}{2}\right)_n}x^j.
\]
Then 
\[
{}_3F_2\left[\begin{array}{rrrr} -\frac{1}{2}, &-n, &\frac{1}{2}+n;& \\ & & &-x\\ & \frac{1}{4},& \frac{3}{4}; \end{array} \right]=1-2 \Psi_n(x).
\]
\end{lemma}
\begin{proof} The following identities are readily verified for $0 \leq k \leq n$.
\begin{eqnarray}
(-1)^k\frac{(-n)_k}{k!}&=&\binom{n}{k}; \label{easy1}\\
\left(\frac{1}{4} \right)_k \left(\frac{3}{4} \right)_k&=&\left(\frac{1}{2} \right)_{2k}\frac{1}{2^{2k}}; \label{easy2}\\
2^k \left(-\frac{1}{2} \right)_k&=&-\frac{(2k-2)!}{2^{k-1}(k-1)!}; \qquad \text{and} \label{easy3}\\
\frac{\left(\frac{1}{2}+n \right)_k}{\left(\frac{1}{2} \right)_{2k}}&=&\frac{\left(\frac{1}{2}+2k \right)_{n-k}}{\left(\frac{1}{2} \right)_n}. \label{easy4}
\end{eqnarray}
With these in hand, we may now compute directly.
\begin{eqnarray*}
{}_3F_2\left[\begin{array}{rrrr} -\frac{1}{2}, &-n, &\frac{1}{2}+n;& \\ & & &-x\\ & \frac{1}{4},& \frac{3}{4}; \end{array} \right]&=&\\
&=& \sum_{k=0}^{\infty} \frac{\displaystyle{\left(-\frac{1}{2} \right)_k(-n)_k\left(n+\frac{1}{2} \right)_k}}{\displaystyle{\left( \frac{1}{4}\right)_k \left(\frac{3}{4} \right)_kk!}} (-x)^k\\
&=&1+\sum_{k=1}^n \frac{\displaystyle{(-1)^k \binom{n}{k} \left(n+\frac{1}{2} \right)_k\left(-\frac{1}{2} \right)_k}}{\displaystyle{\left(\frac{1}{2} \right)_{2k} \frac{1}{2^{2k}}}}(-x)^k\\
&=&1-\sum_{k=1}^n (-1)^k \binom{n}{k}\frac{\displaystyle{\left(\frac{1}{2}+2k \right)_{n-k}2^k(2k-2)!}}{\displaystyle{\left(\frac{1}{2} \right)_{n}2^{k-1}(k-1)!}}(-x)^k\\
&=& 1-2\sum_{k=1}^n \binom{n}{k}\frac{\displaystyle{\left(\frac{1}{2}+2k \right)_{n-k}(2k-2)!}}{\displaystyle{\left(\frac{1}{2} \right)_{n}(k-1)!}}x^k\\
&=&1-2 \Psi_n(x),
\end{eqnarray*}
where the second equality uses equations (\ref{easy1}) and (\ref{easy2}), while the third equality employs equations (\ref{easy3}) and (\ref{easy4}). 
\end{proof}

\begin{lemma}\label{catalanlemma} Let $\displaystyle{C_n:=\frac{1}{n+1}\binom{2n}{n}}$ denote the $n^{th}$ Catalan number. For $n \in \mathbb{N}^{\geq 1}$ the following equality holds:
\begin{eqnarray*}
&&-\frac{C_{n-1}}{3 \cdot 2^{2n-2} \left(\frac{5}{2}\right)_{2n-2}}\\
&=&\frac{1}{2^{2n} \left(\frac{1}{2} \right)_{2n}}\left(2n\frac{(-1)^n \left( \frac{1}{2}\right)_n}{n!}+ \sum_{j=1}^{n-1} \frac{C_{j-1}}{3 \cdot 2^{2j-2} \left( \frac{5}{2} \right)_{2j-2}}\frac{(-1)^{n-j} \left(\frac{1}{2}\right)_{n+j}2^{2j}}{(n-j)!}\right).
\end{eqnarray*}
\end{lemma} 
\begin{proof} Note that the statement of the lemma is equivalent to 
\begin{equation}\label{binomlemma}
0=2n \frac{(-1)^n \left(\frac{1}{2} \right)_{n}}{n!}+\sum_{j=1}^n \frac{C_{j-1} (-1)^{n-j} \left( \frac{1}{2}\right)_{n+j}}{\left(\frac{1}{2} \right)_{2j} (n-j)!}, \qquad \forall n \in \mathbb{N}^{\geq 1},
\end{equation}
or
\begin{equation}\label{binomlemma2}
0=2n+\Psi_n(-1) \qquad \forall n \in \mathbb{N}^{\geq 1},
\end{equation}
where $\Psi_n(x)$ is as in Lemma \ref{appendix}. Combining the results of Lemma \ref{hypergeometric} and \ref{appendix} gives
\[
1-2 \Psi_n(-1)=4n+1, \qquad \forall n \in \mathbb{N}^{\geq 1},
\]
or equivalently, $\Psi_n(-1)=-2n$ for $n\geq 1$. The proof is complete.
\end{proof}
We now prove the main theorem of the section.
\begin{theorem}\label{linear} Consider the operator $T:\R[x] \to \R[x]$ given by 
\[
T[\Le{k}]=(k+c) \Le{k} \qquad \textrm{for} \qquad k=0,1,2,3,\ldots, c\in \R.
\]
 If we write $\displaystyle{T=\sum_{k=0}^{\infty} T_k(x)D^k}$, then 
\begin{equation}\label{tk(0)}
T_k(0)=\left\{ \begin{array}{ll} 0 & \textrm{if $k$ is odd} \\ 
c & \textrm{if $k=0$} \\
\displaystyle{-\frac{C_{n-1}}{3 \cdot 2^{2n-2} \left( \frac{5}{2} \right)_{2 n-2}}} & \textrm{if $k=2n$}, \qquad (n \geq 1) \end{array}\right.
\end{equation}
where $\displaystyle{C_n}$ denotes the n${}^{th}$ Catalan number.
\end{theorem}
\begin{proof} The following facts about Legendre polynomials are known explicitly, or follow easily from basic properties (see for example \cite[p.157-158]{rainville}):
\begin{itemize}
\item[(i)] 
\[
\Le{n}=\frac{2^n \left(\frac{1}{2}\right)_n x^n}{n!}+\pi_{n-2}, \qquad (n \geq 0),
\]
where $\pi_{n-2}$ is a polynomial of degree $n-2$ in $x$;
\item[(ii)] 
\[
\mathfrak{Le}_{2n+1}(0)=0\qquad (n \geq 0);
\]
\item[(iii)] 
\[
\mathfrak{Le}_{2n}(0)=\frac{(-1)^n \left(\frac{1}{2} \right)_n}{n!}\qquad (n \geq 0);
\]
\item[(iv)] For $0\leq j \leq n$
\[
D^{2j} \mathfrak{Le}_{2n}(x)\Big|_{x=0}=\frac{(-1)^{n-j} \left(\frac{1}{2}\right)_{n+j}2^{2j}}{(n-j)!},
\]
while 
\[
D^{2j} \mathfrak{Le}_{2n+1}(x)\Big|_{x=0}=0 \quad  \text{for all} \quad j,n \geq 0,
\]
 simply because Legendre polynomials with odd index are odd.
\end{itemize}
The {\it mutatis mutandis} proof of Proposition 29 in \cite{andrzej} demonstrates that the coefficient polynomials $T_k(x)$ of the linear operator given in Theorem 4 can be computed recursively as
\begin{eqnarray*}
T_0(x)&=&T[1], \qquad \text{and} \\
T_k(x)&=&\frac{1}{2^k \left(\frac{1}{2} \right)_k}\left(T[\Le{k}] -\sum_{j=0}^{k-1} T_j(x)D^j[\Le{k}]\right) \qquad (k=1,2,3,\ldots).
\end{eqnarray*}
It is now easy to verify that $T_0(x)=c$, $T_1(x)=x$ and $T_2(x)=-\frac{1}{3}$, and the proposed values of $T_k(0)$ follow readily for $k=0,1,2$.
Proceeding by induction we assume that $T_j(0)$ is given by equation (\ref{tk(0)}) for $0 \leq j \leq k-1$ for some $k \geq 1$. If $k$ is odd, the second part of fact (iv) above yields
\begin{eqnarray*}
T_k(0)&=&\frac{1}{2^k \left(\frac{1}{2} \right)_k}\left[(k+c) \mathfrak{Le}_k(0)-\sum_{j=0}^{k-1} T_j(x)D^j[\Le{k}] \Bigg|_{x=0}\right]\\
&=&\frac{1}{2^k \left(\frac{1}{2} \right)_k}\left[-\sum_{j=0}^{\frac{k-1}{2}} T_{2j}(x)D^{2j}[\Le{k}] \Bigg|_{x=0}\right]\\
&=&0.
\end{eqnarray*}
On the other hand, if $k$ is even, writing $k=2n$ and using the first part of fact (iv) gives
\begin{eqnarray*}
T_k(0)&=&\frac{1}{2^{2n} \left(\frac{1}{2} \right)_{2n}}\left[(2 n+c) \frac{(-1)^n \left(\frac{1}{2} \right)_n}{n!}
-\sum_{j=0}^{k-1} T_j(x)D^j[\Le{k}] \Bigg|_{x=0}\right]\\
&=&\frac{1}{2^{2n} \left(\frac{1}{2} \right)_{2n}}\left[(2 n) \frac{(-1)^n \left(\frac{1}{2} \right)_n}{n!}-\sum_{j=1}^{k-1} T_j(x)D^j[\Le{k}] \Bigg|_{x=0}\right]\\
&=&\frac{1}{2^{2n} \left(\frac{1}{2} \right)_{2n}}\left[(2 n) \frac{(-1)^n \left(\frac{1}{2} \right)_n}{n!}-\sum_{j=1}^{\frac{k-2}{2}} T_{2j}(x)D^{2j}[\Le{k}] \Bigg|_{x=0}\right]\\
&=&\frac{1}{2^{2n} \left(\frac{1}{2} \right)_{2n}}\left[(2 n) \frac{(-1)^n \left(\frac{1}{2} \right)_n}{n!}+\sum_{j=1}^{n-1} \frac{C_{j-1}}{3 \cdot 2^{2j-2} \left( \frac{5}{2} \right)_{2j-2}}\frac{(-1)^{n-j} \left(\frac{1}{2}\right)_{n+j}2^{2j}}{(n-j)!}
\right]\\
&=&-\frac{C_{n-1}}{3 \cdot 2^{2n-2} \left(\frac{5}{2}\right)_{2n-2}},
\end{eqnarray*}
where the last equality is the result of Lemma \ref{catalanlemma}. 
\end{proof}
Let $T$ be the operator corresponding to the Legendre sequence $\seq{k+c}$. Recall that the symbol of $T$ is given by
\[
G_T(-x,y)=\sum_{k=0}^{\infty} \frac{(-1)^k T[x^k]y^k}{k!},
\]
and that $T$ is reality preserving (i.e. $\seq{k+c}$ is a Legendre multiplier sequence) if and only if either $G_T(-x,y)$ or $G_T(x,y)$ belongs to $\mathcal{L-P}_2(\R)$, since the sequence under consideration is non-trivial. Following \cite{bo}, we expand $G_T(-x,y)$ and $G_T(x,y)$ as a series in powers of $x$. By Theorem \ref{linear} the constant term in both of these expansions is 
\[
f(y):=\displaystyle{c-\sum_{k=1}^{\infty} \frac{C_{k-1}\cdot y^{2k}}{3 \cdot 2^{2k-2} \left(\frac{5}{2} \right)_{2k-2}}}.
\]
Thus $f(y) \in \mathcal{L-P}$ if either $G_T(-x,y)$ or $G_T(x,y)$ were in $\mathcal{L-P}_2(\R)$, since we obtain $f(y)$ from either $G_T(-x,y)$ or $G_T(x,y)$ by applying the non-negative multiplier sequence $\{1,0,0,0,\ldots\}$ acting on $x$, which preserves the class $\mathcal{L-P}_2(\R)$ (see \cite{bb} and \cite{branden}). We shall now demonstrate that $f(y)$ is an entire function  which does not belong to the Laguerre-P\'olya class, and hence $\seq{k+c}$ is not a Legendre multiplier sequence for any $c \in \R$.
\begin{proposition} Let $c \in \R$. Then 
\[
f(y)=\displaystyle{c-\sum_{k=1}^{\infty} \frac{C_{k-1}\cdot y^{2k}}{3 \cdot 2^{2k-2} \left(\frac{5}{2} \right)_{2k-2}}}
\]
is an entire function which does not belong to $\mathcal{L-P}$.
\end{proposition}
\begin{proof} Consider the change of variables $x=y^2$ and the function
\[
\widetilde{f}(x)=\displaystyle{c-\frac{4}{3}\sum_{k=1}^{\infty} \frac{C_{k-1}\cdot x^k}{2^{2k} \left(\frac{5}{2} \right)_{2k-2}}}=c-\frac{4}{3}\sum_{k=1}^{\infty} a_k x^k.
\]
Since 
\[
(\star) \qquad\lim_{k \to \infty} \frac{a_{k+1}}{a_k}=\lim_{k \to \infty} \frac{2(2k-1)}{k+1}\cdot \frac{1}{(5+2(2k-2))(5+2(2k-1))}=0,
\]
$\widetilde{f}(x)$ is entire. The existence of the limit in $(\star)$ implies that $\displaystyle{\lim_{k \to \infty} \sqrt[k]{a_k}=0}$ as well, and hence $f(y)$ is also entire. 
\newline \indent It remains to show that $f(y) \notin \mathcal{L-P}$. To this end, we first demonstrate that $\widetilde{f}(x) \notin \mathcal{L-P}$. Writing $d_k=k! a_k$ we can express $\widetilde{f}(x)$ as
\[
\widetilde{f}(x)=c-\frac{4}{3}\sum_{k=1}^{\infty} \frac{d_k}{k!}x^k.
\]
By Theorem \ref{cvish} and the comments thereafter, if $\widetilde{f}(x)$ were to belong to $\mathcal{L-P}$, we would have $L_1(x, \widetilde{f}^{(k)}) \geq 0$ for all $k=0,1,2,\ldots$ and $x \in \R$. In particular, $\displaystyle{L_1(0,\widetilde{f}')=\frac{16}{9}\left(d_2^2-d_{3}d_{1}\right) \geq 0}$ would hold. A quick calculation reveals that 
\[
d_2^2-d_3d_1=-\frac{1}{80850}<0,
\]
establishing that $\widetilde{f}(x) \notin \mathcal{L-P}$. Suppose now that $f(y) \in \mathcal{L-P}$. By virtue of being an even function, $f(y)$ has the factorization
\[
f(y)=ce^{-a y^2} \prod_{k=1}^{\omega} \left( 1-\frac{y^2}{x^2_k} \right),
\]
where $a \geq 0$ and $x_k \in \R \setminus\{0\}$, $0 \leq \omega \leq +\infty$, and $\displaystyle{\sum 1/x_k^2< +\infty}$. Replacing $y^2$ by $x$ would yield $\tilde{f}(x) \in \mathcal{L-P}$, a contradiction. We conclude that $f(y) \notin \mathcal{L-P}$, and our proof is complete.
\end{proof}

\vspace{.25in}
\section{Cubic Legendre multiplier sequences}\label{cubics}
In this section we establish the non-existence of cubic Legendre multiplier sequences. Without loss of generality we may consider sequences interpolated by monic polynomials. Since every such cubic polynomial can be written as $(k^2+\alpha k+\beta)(k+c)$ for some real triple $(\alpha, \beta, c)$, one may wish to proceed based on whether or not the quadratic factor in the product is itself a Legendre multiplier sequence. It turns out that such case analysis is more than one needs: we can handle all cubic sequences at once. We begin with two preparatory results. 
\begin{lemma}\label{cubicCMSclass}
Suppose $T=\seq{k^3+ak^2+bk+c}$ is a sequence of non-negative terms. If $T$ is a classical multiplier sequence, then $a\geq -3$, $a+b \geq -1$ and $c \geq 0$.
\end{lemma}

\begin{proof}
By part (3) of Theorem \ref{Polya}, $T$ is a \cms ~ if and only if 
\begin{eqnarray*}
T[e^x] &=& \sum_{k=0}^{\infty}(k^3+ak^2+bk+c)\frac{x^k}{k!}\\
&=&e^x \left( x^3+(a+3)x^2+(a+b+1)x+c\right) \in \mathcal{L-P^+}.
\end{eqnarray*}
In particular, the coefficients of the polynomial 
\[
p(x)=x^3+(a+3)x^2+(a+b+1)x+c
\]
must all be non-negative. The claim follows.
\end{proof}
\begin{lemma}\label{oppsign}(\cite[Lemma 3, p. 337]{levin}) If all zeros of the real polynomial 
\[
h(x)=c_0+c_1x+\cdots+c_nx^n \quad (c_n \neq 0)
\]
are real, $c_0 \neq 0$ and $c_p=0$ $(0<p<n)$, then $c_{p-1}c_{p+1}<0$.
\end{lemma}
We are now ready to state and prove the main theorem of the section.
\begin{theorem}\label{nocubics} The sequence $\seq{k^3+ak^2+bk+c}$ is not a Legendre multiplier sequence for any real triple $(a,b,c)$. 
\end{theorem}
\begin{proof}
Denote by $T_{a,b,c}$ the operator associated to the Legendre sequence $\seq{k^3+ak^2+bk+c}$. By Lemma \ref{cubicCMSclass}, in order  for $\seq{k^3+ak^2+bk+c}$ to be a classical multiplier sequence we must have $a\geq -3$, $a+b \geq -1$ and $c \geq 0$.
Consider now the action of $T_{\alpha,\beta,c}$ on the two polynomials
\begin{eqnarray*}
p_1(x)&=&x^5 \Le{3}\\
&=&\frac{64}{1287} \Le{8}+\frac{152}{693} \Le{6}+\frac{372}{1001} \Le{4}+\frac{205}{693} \Le{2}+\frac{4}{63}
\end{eqnarray*}
and 
\begin{eqnarray*}
p_2(x)&=&x^5 \Le{5}\\
&=&\frac{2016}{46189} \Le{10}+\frac{4816}{24453} \Le{8}+\frac{4078}{11781} \Le{6}+\frac{291}{1001} \Le{4}+\frac{1000}{9009}\Le{2}+\frac{8}{693}.
\end{eqnarray*}
Computing
\begin{eqnarray*}
18018 \cdot T_{a,b,c}[p_1(x)]&=&\sum_{k=0}^4 q_{2k}(a,b,c)x^{2k}
\end{eqnarray*}
we find that
\begin{eqnarray*}
q_0(a,b,c)&=&16(-121+46a-46 b), \qquad \text{and}\\
q_4(a,b,c)&=&630(15724+1226 a+61 b),
\end{eqnarray*}
with the restrictions on $a,b$ and $c$ implying directly that $q_4(a,b,c) >0$ for all real triples $(a,b,c)$ under consideration. If $q_2(a,b,c)=0$, then reversing coefficients, and taking four derivatives of $T_{a,b,c}[p_1(x)]$ (both of which operations preserve the reality of zeros) results in a polynomial with non-real zeros. If $q_2(a,b,c) \neq 0$, then in light of Lemma \ref{oppsign}, a necessary condition for $T_{a,b,c}[p_1(x)]$ to have only real zeros is that
\[
(\dag) \qquad q_0(a,b,c)=16(-121+46a-46 b) \geq 0.
\]
We now turn our attention to $T_{a,b,c}[p_2(x)]$. If we write 
\[
23279256 \cdot T_{a,b,c}[p_2(x)]=\sum_{k=0}^5 w_{2k}(a,b,c)x^{2k},
\]
then
\begin{eqnarray*}
w_0(a,b,c)&=&16(-641+806 a-806 b), \qquad \text{and}\\
w_4(a,b,c)&=&-630(38840980+2015774 a+62731 b),
\end{eqnarray*}
with Lemma \ref{cubicCMSclass} implying that $w_4(a,b,c)<0$ for all admissible triples $(a,b,c)$. Considerations identical to those above imply that either $T_{a,b,c}[p_2(x)]$ has non-real zeros, or the inequality
\begin{eqnarray*}
w_0(a,b,c)=16(-641+806 a-806 b)\leq 0
\end{eqnarray*}
must hold. Combining inequalities $(\dag)$ and $(\ddag)$ we obtain
\[
-\frac{121}{46}+a \geq b \geq -\frac{641}{806}+a,
\]
a clear impossibility. We conclude that $T_{a,b,c}$ cannot simultaneously preserve the reality of the zeros of $x^5 \Le{3}$ and $x^5 \Le{5}$. Whence $\seq{k^3+ak^2+bk+c}$ is not a Legendre multiplier sequence for any real triple $(a,b,c)$.
\end{proof}
\begin{remark} Theorem \ref{nocubics} yields yet another proof of the non-existence of linear Legendre multiplier sequences by the following considerations. If $T_1, T_2$ are Legendre multiplier sequences, then so is $T_1 T_2$. Since $\{k^2+k+\beta\}$ is a Legendre multiplier sequence whenever $\beta \in [0,1]$, the existence of linear Legendre multiplier sequences would immediately imply the existence of cubic Legendre multiplier sequences, contradicting Theorem \ref{nocubics}. 
\end{remark}
\section{Open problems}
The following is a list of open problems motivated by the preceding results. These questions are not only related to the classification of Legendre multiplier sequences but also to some general properties of reality preserving linear operators $\displaystyle{T=\sum_{k=0}^{\infty}T_k(x)D^k}$ on $\R[x]$, properties which are captured in the coefficient polynomials $T_k(x)$.
\subsection{Higher order Legendre sequences} The characterization of polynomials with degree four or higher which interpolate Legendre multiplier sequences remains open. Using computational techniques as in Section \ref{cubics} quickly turns intractable with the increasing number of parameters. In addition, one has to judiciously select ``test polynomials'' in order for this method to succeed succinctly. The polynomials
\[
p(n,k)=x^k \Le{n}
\] 
mimic properties of the test polynomials $(1+x)^n$ for classical multiplier sequences in that they have zeros of high multiplicity away from the zeros of the basis polynomials. As such, we were able to use just a couple test polynomials to demonstrate the non-existence of cubic Legendre multiplier sequences. On the downside, the degrees of these polynomials are high and we believe that the degrees of the test polynomials would have to increase if one would want to eliminate sequences interpolated by higher order polynomials. 
\subsection{Monotone operators} We call an operator $\displaystyle{T=\sum_{k=0}^{\infty}T_k(x)D^k}$ monotone if $\displaystyle{\deg T_k(x) \geq \deg T_{k-1}(x)}$ for all $k=1,2,\ldots$. The operator corresponding to the linear Legendre sequence $\seq{k+c}$ is given by 
\[
T=c+xD-\frac{1}{3}D^2+\frac{2}{15}xD^3+\sum_{k=4}^{\infty}T_k(x)D^k,
\]
whereas the operator corresponding to the Legendre sequence $\seq{k^2+\alpha k+\beta}$, $\alpha \neq 1$ is given by 
\begin{eqnarray*}
T&=&\beta+(1+\alpha)xD-\left(\frac{2+\alpha-3x^2}{3}\right)D^2+\frac{2}{15}(\alpha-1)xD^3-\frac{(\alpha-1)(1+4x^2)}{105}D^4\\
&+&(\alpha-1)\sum_{k=5}^{\infty} T_k(x)D^k.
\end{eqnarray*}
Neither sequence is a Legendre multiplier sequence, and neither operator is monotone. We believe these facts to be related, and give the following 
\begin{conjecture} Suppose $\displaystyle{T=\sum_{k=0}^{\infty}T_k(x)D^k}$ is an infinite order differential operator. If $T$ is not monotone, then $T$ is not reality preserving.
\end{conjecture}
\noindent Should this conjecture be true, one could then try to prove that if $\seq{\gamma_k}=\seq{p(k)}$ where $\deg p$ is odd, and $\seq{\gamma_k}$ is a Legendre sequence, then the operator corresponding to the sequence $\seq{\gamma_k}$ is an infinite order differential operator which is not monotone.
\subsection{Using the symbol of the operator} Our approach used in Section \ref{linears} could be extended to treat sequences interpolated by higher order polynomials. Piotrowski (\cite{andrzej}) and Forg\'acs and Piotrowski (\cite{tomandrzej}) give explicit representations of the coefficient polynomials $T_k(x)$ of classical, and Hermite diagonal operators respectively. In both cases the $T_k(x)$s are given in terms of the reverses of the Jensen polynomials associated to the sequence $\seq{\gamma_k}$. If a sequence $\seq{\gamma_k}$ is interpolated by a polynomial, then only finitely many of these reverse Jensen polynomials are non-zero. This means that an analog of Theorem \ref{linear} would need the identification of only finitely many sequences, one for each inverse Jensen polynomial involved in the $T_k(x)$s, in order to explicitly determine the sequence $\seq{T_k(0)}$. With this sequence in hand, one could carry out steps analogous to those in Section \ref{linears} to establish the non-existence of Legendre multiplier sequences interpolated by polynomials of degree greater than three. 

\medskip 

\noindent {\bf Acknowledgement.} \ We would like to thank George Csordas for many stimulating discussions and guiding insights, and the anonymous referee for numerous suggestions improving the exposition and streamlining the proofs of Lemma 2 and Theorem 5.

\bigskip

\noindent ${}^{\dag}$  Department of Mathematics\\
5245 N. Backer Ave, M/S PB 108\\
California State University, Fresno 93740-8001

\bigskip

\noindent ${}^{\ddag}$  UR Mathematics\\
915 Hylan Building\\
University of Rochester, RC Box 270138\\
Rochester, NY 14627

\bigskip

\noindent ${}^{\S}$  Department of Mathematics and Statistics\\
College of Sciences and Mathematics\\
Auburn University\\
221 Parker Hall\\
Auburn, AL 36849

\bigskip

\noindent ${}^{\star}$ Mathematics Department \\
Davidson College\\
Box 7129\\
Davidson, NC 28035


\begin{thebibliography}{1}

\bibitem{Blakeman} {\it K. Blakeman, E. Davis, T. Forg\'acs,} and {\it K. Urabe}, On Legendre multiplier sequences, Missouri J. Math. Sci. \textbf{24}, (2012), 7-23. 

\bibitem{bb} {\it J. Borcea} and {\it P. Br\"and\'en}, Multivariate P\'olya-Schur classification problems in the Weyl algebra, Proc. London Math. Soc., {\bf 101} (2010), 73-104.

\bibitem{branden} {\it P. Br\"and\'en}, The Lee-Yang and P\'olya-Schur programs. III. Zero-preservers on Bargmann-Fock spaces, Amer. J. Math, to appear. 

\bibitem{bo}{\it P. Br\"and\'en} and {\it E. Ottergren}, A characterization of multiplier sequences for generalized Laguerre bases, Constructive Approximation, to appear.

\bibitem{cvarga} {\it G. Csordas} and {\it R.S. Varga}, Necessary and sufficient conditions and the Riemann Hypothesis, Adv. in Appl. Math., {\bf 11}, (1990), 328-357.

\bibitem{cv} {\it G. Csordas} and {\it A. Vishnyakova}, The generalized Laguerre inequalities and functions in the Laguerre-P\'olya class, Cent. Eur. J. Math. {\bf 11}, 2013, 1643-1650.

\bibitem{tomandrzej} {\it T. Forg\'acs} and {\it A. Piotrowski}, Hermite multiplier sequences and their associated operators. arXiv:1312.6187 [math.CV]

\bibitem{tomandrzej2} {\it T. Forg\'acs} and {\it A. Piotrowski},  Multiplier sequences for generalized Laguerre bases, The Rocky Mountain J. of Math., {\bf 43}, 2013, 1141-1159. 

\bibitem{levin} {\it B. Ja. Levin}, Distributions of Zeros of Entire Functions, Transl. Math. Mono. {\bf (5)}, Amer. Math. Soc., Providence, RI 1964; revised ed. 1980.

\bibitem{obre} {\it N. Obreschkoff}, Verteilung und Berechnung der Nullstellen Reeller Polynome, Veb Deutscher Verlag der Wissenschaften, Berlin, 1963.

\bibitem{andrzej} {\it A. Piotrowski}, Linear Operators and the Distribution of Zeros of Entire Functions, Ph.D. Dissertation. University of Hawai\textquoteleft i, 2007.

\bibitem{Polya} {\it G. P\' olya} and {\it J. Schur}, \"{U}ber zwei Arten von Faktorenfolgen in der Theorie der algebraischen Gleichungen, J. Reine Angew. Math. \textbf{144}, (1914), 89-113. 

\bibitem{rainville} {\it E.D. Rainville}, Special Functions, The Macmillan Company, New York, 1960.


\bibitem{yoshi} {\it R. Yoshida}, Linear and non-linear operators, and the distribution of zeros of entire functions, Ph.D. dissertation. University of Hawai\textquoteleft i, 2013. 

\end{thebibliography}
\end{document}